\documentclass{amsart}
\usepackage[T1]{fontenc}
\usepackage[utf8]{inputenc}%pour les accents, sinon latin9
\usepackage{amsmath,amsthm,amsfonts,amscd,amssymb,eucal,latexsym,mathrsfs}
\usepackage{stmaryrd}
\usepackage{enumerate}
\usepackage{hyperref}
\usepackage[all]{xy}
\usepackage{etoolbox}

\usepackage{tikz,tikz-cd}
\usetikzlibrary{shapes.geometric}
\usetikzlibrary{arrows}

\usetikzlibrary{positioning}
\usepackage{float}
\usepackage{MnSymbol}
\usetikzlibrary{matrix}

\newcommand{\toc}{\tableofcontents}

\theoremstyle{plain}
\newtheorem{theorem}{Theorem}[section]
\newtheorem*{theorem*}{Theorem}

\newtheorem*{corollary*}{Corollary}

\newtheorem{lemma}[theorem]{Lemma}

\newtheorem{claim}[theorem]{Claim}
\theoremstyle{definition}
\newtheorem{remark}[theorem]{Remark}

\newtheorem{definition}[theorem]{Definition}
\newtheorem*{definition*}{Definition}

\usetikzlibrary{calc,graphs}
\usepackage{xcolor}
\usetikzlibrary{arrows,decorations.pathmorphing}

\newcommand{\p}{\varphi}

\newcommand{\ZI}{\mathbb{Z}}

\DeclareMathOperator{\lra}{\leftrightarrow}

\newcommand{\ts}{\textsection}

\newcommand{\ip}[1]{\langle#1\rangle} % \ip{a,b} gives us <a,b>

\newcommand{\Aut}{\mathrm{Aut}}

\title{The Hamiltonian surface of $\Aut(F_2)$}
\author{Sylvain Barr\'e}
\author{Mika\"el Pichot}
\address{Sylvain Barr\'e, UMR 6205, LMBA, Université de Bretagne-Sud,BP 573, 56017, Vannes, France}\email{Sylvain.Barre@univ-ubs.fr}
\address{Mika\"el Pichot, McGill University, 805 Sherbrooke St W., Montr\'eal, QC H3A 0B9, Canada}\email{mikael.pichot@mcgill.ca}
\begin{document}

\setcounter{tocdepth}{1}

\begin{abstract} 
The group of automorphisms of the free group on two generators is known to act geometrically, in an essentially  unique way, on a 2-dimensional CAT(0) space called the Brady complex.  

In the present paper, we prove that the Brady complex $X$ contains precisely two Hamiltonian surfaces $\Sigma_1$ and $\Sigma_2$. By this we mean a surface in $X$ which visits every vertex and every edge precisely once.

Furthermore, there exist an isomorphism $\p\colon G_1\to G_2$ between two finite index subgroups $G_1$ and $G_2$ in $\Aut(F_2)$, respectively stabilizing the surfaces $\Sigma_1$ and $\Sigma_2$, and an isometric isomorphism $\psi\colon \Sigma_1\to \Sigma_2$, such that $\psi(g_1x_1)=\p(g_1)\psi(x_1)$ for every $x_1\in \Sigma_1$ and every $g_1\in G_1$. 

In other words, the Brady complex contains a unique Hamiltonian surface up to virtually equivariant isomorphism. We call this surface the Hamiltonian surface of $\Aut(F_2)$. It can be viewed as a covering space of a compact hyperbolic surface of genus 2 which can be described explicitly.
\end{abstract}

\maketitle

\section{Introduction}

The group of automorphisms of the free group on two generators is known to act geometrically, essentially in a unique way, on a CAT(0) space $X$ of dimension 2, called the Brady complex (see \cite[Theorem 1]{cp}).  
In a recent paper \cite{torus}, we gave a new construction of the complex $X$  starting from a metric 2-torus, in three steps, as follows:

\begin{enumerate}
\item start with a metric torus $T$ obtained by identifying isometrically the sides of a parallelogram of size $6\times 5$, subdivided into 30 Euclidean lozenges with sides of length 1   (using an appropriate partial Dehn twist, see \cite[Fig.\ 1]{torus});

\item pinch the vertices of the torus $T$ in pairs, by exhibiting an appropriate involution $\sigma$ of the vertex set of $T$---this results in a pinched torus $T/\ip{\sigma}$;

\item fill in the pinched torus $T/\ip{\sigma}$ by gluing 20 equilateral triangles  on its edge set  (we note here that since $T$ has 30 lozenges, $T/\ip{\sigma}$ contains 60 edges). 
\end{enumerate}

\noindent It can be verified (see \cite[\ts 2]{torus}) that the universal covering space $X$  of the compact cellular space $V$  defined by this three steps procedure is isomorphic to the Brady complex, and that its fundamental group is isomorphic to a subgroup of finite index in the group $\Aut(F_2)$ of automorphisms of the free group on two generators.

The metric torus $T$ maps to $V$ in the obvious way onto a pinched torus, namely $T/\ip\sigma$, which can be lifted to the universal cover $X$, as an infinite embedded surface. Let us call this surface $\Sigma$. It is to verify that $\Sigma$ is a connected surface which contains every vertex of $X$, and has infinitely many double points (it has one such point at every vertex). The double points arise from Step (2) in the construction. The present paper provides an alternative construction of $X$  in which the pinching step is not required, and which removes the double points. 

The space $V$ was discovered by using surgery techniques, applied to the group $\Aut(F_2)$, in particular, through an explicit description of some collars and group cobordisms for this group  (see \cite[\ts 6]{torus}). 
The specific choice of the metric torus $T$ in Step (1) turns out to be quite constrained; in particular, it was shown in \cite{torus} how to do a similar construction starting from a torus of size $6\times n$, for a arbitrary integer $n\geq 1$, but the use of other tori, or even the role of the value $m=6$ in this construction, remained unclear.
The ideas developed in the present paper are entirely independent of surgery. Rather than classifying collars and group cobordisms, we shall exhibit a canonical (up to equivariant isomorphism) surface directly in the universal cover $X$ itself, proceed to describe a compact quotient, which then can be used as a replacement to the torus in Step (1). 

In order to explain how a canonical, hyperbolic surface exit in a CAT(0) space, we shall begin with the observation that, because of the filling construction in Step (3) above, the surface $\Sigma$ (defined to be the lift of the pinched torus $T/\ip\sigma$ to the universal cover of the Brady complex $X$) has the interesting property of ``enveloping'' the CAT(0) complex $X$ in the following sense: 

\begin{definition}\label{D - enveloping surface}
We call \emph{enveloping surface} in a 2-complex $X$ a (connected) surface without boundaries which visits every edge precisely once.  
\end{definition}

Such a surface necessarily visits every vertex of $X$ (possibly multiple times). 
We mention that, in the construction above, once the enveloping surface $\Sigma$ is given, the filling-in procedure (i.e., Step (3)) is essentially canonical: it is a ``systolic filling'', which means that an equilateral triangle is glued for every  cycle of length 3 in $\Sigma$, where the  cycles of length 3 in $\Sigma$ are the systoles.

Roughly speaking, the notion of enveloping surface introduced in Def.\  \ref{D - enveloping surface} can be viewed as a 2-dimensional analog of the notion of a Hamilonian cycle in a finite graph. Although enveloping surfaces do not seem to have appeared in the literature thus far,  we note that the related (but somewhat distant) concept of Hamiltonian cycles, in a \emph{finite hypergraph}, is of course  well-known and well-studied. In the context  of infinite spaces, another related (and also somewhat distant) concept is that of a Hamiltonian, semi-infinite (or bi-infinite), path in an infinite graph.   

In our context, which concerns mainly non-positively curved 2-dimensional cellular complexes, we shall prefer to reserve the word ``Hamiltonian'' to surfaces which have no double point, and visit every edge, as well as every vertex, precisely once: 

\begin{definition}
We call \emph{Hamiltonian surface} in a 2-complex $X$ an enveloping surface which has no multiple vertex.  
\end{definition}

In other words, we shall call Hamiltonian surface  an embedded enveloping surface. It was recalled above that the enveloping surface $\Sigma$ in the Brady complex has infinitely many double points, and therefore, it is not a Hamiltonian surface of $X$ in this sense. 

\bigskip

Given this terminology, our main theorem in the present paper states the following.

\begin{theorem}\label{T - main theorem}
The Brady complex $X$ contains precisely two Hamiltonian surfaces.
\end{theorem}

We observe that a Hamiltonian surface in $X$ is necessarily hyperbolic of infinite genus. By analogy with graph theory, we call the group $\Aut(F_2)$ itself a ``Hamiltonian group''. Namely, it is Hamiltonian the sense that it acts geometrically on a CAT(0) 2-complex which contains a Hamiltonian surface; similarly, we shall call the Brady complex $X$ itself, a \emph{Hamiltonian CAT(0) complex}.  

Furthermore, we prove that the two surfaces in Th.\ \ref{T - main theorem} are \emph{periodic}, in the sense that they are acted upon freely by finite index subgroups in $\Aut(F_2)$. In addition, we prove that $X$ contains a unique surface up to isomorphism, showing that the group $\Aut(F_2)$ is in fact ``uniquely Hamiltonian'' in the following sense:

\begin{theorem}\label{T - isomorphic}
Let $\Sigma_1$ and $\Sigma_2$ denote the two Hamiltonian surfaces of $X$. Then there exists finite index subgroups $G_1$ and $G_2$ in $\Aut(F_2)$ respectively stabilizing the surfaces $\Sigma_1$ and $\Sigma_2$, a group isomorphism $\p\colon G_1\to G_2$ between  these groups, and an isometric isomorphism $\psi\colon \Sigma_1\to \Sigma_2$ between the Hamiltonian surfaces, such that $\psi(g_1x_1)=\p(g_1)\psi(x_1)$ for every $x_1\in \Sigma_1$ and every $g_1\in G_1$. 
\end{theorem}

Although we have at this point admittedly only begun the study of Hamiltonian surfaces in CAT(0) complexes, the fact that $X$ contains precisely two distinct Hamiltonian surfaces seems to be quite rare. There are  examples of CAT(0) complexes having  no Hamiltonian surface, and CAT(0) complexes having infinitely many. We do not know how to describe CAT(0) complexes containing precisely $n$ surfaces, where $n$ is a fixed integer.

We conclude with the remark that one can state an ``analog'' of the flat closing conjecture for Hamiltonian surfaces: suppose a group $G$ acts geometrically on a CAT(0) 2-complex $X$. If $X$ contains a Hamiltonian surface, does it also contain a periodic Hamiltonian surface?  This holds in the case of the group $G=\Aut(F_2)$, acting on the Brady complex, by the results above.

\toc

\section{The Moebius ladder}

Let $X$ be a CAT(0) 2-complex. For the notion of Hamiltonian surface in $X$ to be interesting, the vertex links of $X$ should all be Hamiltonian graphs; if  not then, clearly, $X$ does not contain any Hamiltonian surface. Thus, the following holds.

\begin{lemma}
Every Hamiltonian surface $\Sigma$ in $X$ is connected; for every $x\in X$, the intersection of  the link  $L_x$ of $x$ in $X$ with $\Sigma$ is a Hamiltonian cycle in $L_x$.  
\end{lemma}

A complex is said to have order 2 if every edge is incident to exactly three faces. The  Brady complex $X$ is a CAT(0) complex of order 2 made of triangles and lozenges (see \cite{autf2puzzles}). It is known  (see \cite{cp}) that $\Aut(F_2)$ acts transitively on the vertex set of $X$ with unique vertex link isomorphic to a cubic graph on eight vertices:

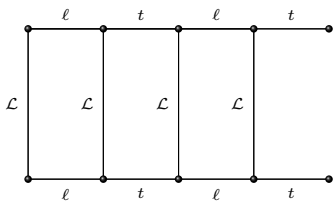
\begin{figure}[H]

\begin{tikzpicture}[shift ={(0.0,0.0)},scale = 1]
\tikzstyle{every node}=[font=\tiny]

\coordinate  (v1) at (0.0,0.0);
\coordinate (v2) at (1.0,0.0);
\coordinate (v3) at (2.0,0.0);
\coordinate (v4) at (3.0,0.0);
\coordinate (v5) at (0.0,2.0);
\coordinate (v6) at (1.0,2.0);
\coordinate (v7) at (2.0,2.0);
%\coordinate[label=right:{$v^-$}] (v8) at (1.0,-1.0);
\coordinate (v8) at (3.0,2.0);
\coordinate (v9) at (4.0,0.0);
\coordinate (v10) at (4.0,2.0);
\draw[solid,thin,color=black,-] (v1) -- (v2);
\draw[solid,thin,color=black,-] (v2) -- (v3);
\draw[solid,thin,color=black,-] (v3) -- (v4);
\draw[solid,thin,color=black,-] (v5) -- (v6);
\draw[solid,thin,color=black,-] (v6) -- (v7);
\draw[solid,thin,color=black,-] (v7) -- (v8);
\draw[solid,thin,color=black,-] (v1) -- (v5);
\draw[solid,thin,color=black,-] (v2) -- (v6);
\draw[solid,thin,color=black,-] (v3) -- (v7);
\draw[solid,thin,color=black,-] (v4) -- (v8);
\draw[solid,thin,color=black,-] (v4) -- (v9);
\draw[solid,thin,color=black,-] (v5) -- (v10);

\draw (v1) -- (v2) node[midway, below] {$\ell$};
\draw (v2) -- (v3) node[midway, below] {$t$};
\draw (v3) -- (v4) node[midway, below] {$\ell$};
\draw (v4) -- (v9) node[midway, below] {$t$};
\draw (v5) -- (v6) node[midway, above] {$\ell$};
\draw (v6) -- (v7) node[midway, above] {$t$};
\draw (v7) -- (v8) node[midway, above] {$\ell$};
\draw (v8) -- (v10) node[midway, above] {$t$};
\draw (v1) -- (v5) node[midway, left] {$\mathcal{L}$};
\draw (v2) -- (v6) node[midway, left] {$\mathcal{L}$};
\draw (v3) -- (v7) node[midway, left] {$\mathcal{L}$};
\draw (v4) -- (v8) node[midway, left] {$\mathcal{L}$};

\shade[ball color=black] (v1) circle (.05);
\shade[ball color=black] (v2) circle (.05);
\shade[ball color=black] (v3) circle (.05);
\shade[ball color=black] (v4) circle (.05);
\shade[ball color=black] (v5) circle (.05);
\shade[ball color=black] (v6) circle (.05);
\shade[ball color=black] (v7) circle (.05);
\shade[ball color=black] (v8) circle (.05);
\shade[ball color=black] (v9) circle (.05);
\shade[ball color=black] (v10) circle (.05);

\end{tikzpicture}
\caption{Moebius ladder on four rungs}\label{Fig Moebius Ladder}
\end{figure}

This figure represents the link of $X$ as a Moebius ladder $L$ on four rungs and eight vertices. The labels on this graph are from \cite[\ts 4]{torus} and represent the angular type of every faces: $t$ for a triangle angle, $\ell$ for a small lozenge angle, and $\mathcal{L}$ for a large lozenge angle. This link is a vertex transitive cubic graph which contains Hamiltonian cycles. The  Hamiltonian cycles are easy to classify:

\begin{lemma} \label{L- Hamiltonian cycles} 
The Moebius ladder $L$ contains precisely five Hamiltonian graphs:
\begin{enumerate}
\item A unique cycle without rungs
\item Four cycles having a pair of missing consecutive rungs
\end{enumerate}
By consecutive rung, we mean rungs connected by two horizontal edges. Furthemore, the rungs in a Hamiltonian cycle in case 2 are at distance 3 (i.e., separated by a pair of paths consisting of three horizontal edges).
\end{lemma}

We should point out that the uniqueness of a Hamiltonian surface a complex of order 2, such as in the Brady complex $X$, cannot result from an analogous statement holding in the links, where the Hamiltonian cycles are never unique. Indeed,  every Hamiltonian cubic graph must contain at least 3 distinct Hamiltonian cycles. This follows by the theorem of Tutte \cite{Tu} that, more precisely, any  given edge must contain an even number of such cycles---and therefore that some vertex must contain at least three.

\begin{proof}
We call rung the vertical edges in Fig.\ \ref{Fig Moebius Ladder}. Let $R$ denote the number of rungs in a Hamiltonian cycle. We distinguish several cases.

\bigskip

$R=0$. 

There exists a unique such Hamiltonian cycle. It is the union of the two horizontal tracks.

\bigskip

$R=1$. 

There are no such cycles, since they would contain the two horizontal tracks entirely. 

\bigskip

More generally, we shall use of the following fact which is elementary to establish.

\begin{claim} If a Hamiltonian cycle misses a rung, then it must contain the four adjacent horizontal edges.
\end{claim}

\bigskip

$R=2$. 

There are four cycles corresponding to  pairs of  missing consecutive rungs; it is easily seen in this cases that the rungs are separated by a pair of horizontal paths of simplicial length $3$.

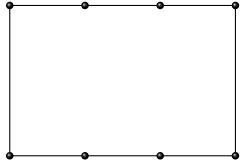
\begin{figure}[H]

\begin{tikzpicture}[shift ={(0.0,0.0)},scale = 1]
\tikzstyle{every node}=[font=\tiny]

\coordinate  (v1) at (0.0,0.0);
\coordinate (v2) at (1.0,0.0);
\coordinate (v3) at (2.0,0.0);
\coordinate (v4) at (3.0,0.0);
\coordinate (v5) at (0.0,2.0);
\coordinate (v6) at (1.0,2.0);
\coordinate (v7) at (2.0,2.0);
%\coordinate[label=right:{$v^-$}] (v8) at (1.0,-1.0);
\coordinate (v8) at (3.0,2.0);
\coordinate (v9) at (4.0,0.0);
\coordinate (v10) at (4.0,2.0);
\draw[solid,thin,color=black,-] (v1) -- (v2);
\draw[solid,thin,color=black,-] (v2) -- (v3);
\draw[solid,thin,color=black,-] (v3) -- (v4);
\draw[solid,thin,color=black,-] (v5) -- (v6);
\draw[solid,thin,color=black,-] (v6) -- (v7);
\draw[solid,thin,color=black,-] (v7) -- (v8);
\draw[solid,thin,color=black,-] (v1) -- (v5);
\draw[solid,thin,color=black,-] (v4) -- (v8);

\shade[ball color=black] (v1) circle (.05);
\shade[ball color=black] (v2) circle (.05);
\shade[ball color=black] (v3) circle (.05);
\shade[ball color=black] (v4) circle (.05);
\shade[ball color=black] (v5) circle (.05);
\shade[ball color=black] (v6) circle (.05);
\shade[ball color=black] (v7) circle (.05);
\shade[ball color=black] (v8) circle (.05);

\end{tikzpicture}
\caption{Hamiltonian cycle with two rungs}
\end{figure}

It follows by the claim above that the missing rungs are consecutive. 

\bigskip

$R=3$. 

There are no such cycles. Indeed, it would have a unique missing rung, and therefore would, by the claim, contain the four horizontal edges attached to this rung, which is clearly impossible.  

\bigskip

$R=4$. 

There are no such cycles. Indeed, it is easily verified that, starting from an arbitrary horizontal edge, the unique non-backtracking path that contains this edge, and all rungs, is not a cycle. This follows because $L$ is a Moebius ladder on four rungs.  
\end{proof}

It is convenient to classify the Hamiltonian cycles into three types as follows:

\begin{figure}[H]

\begin{tikzpicture}[shift ={(0.0,0.0)},scale = 1]
\tikzstyle{every node}=[font=\tiny]

\coordinate  (v1) at (0.0,0.0);
\coordinate (v2) at (1.0,0.0);
\coordinate (v3) at (2.0,0.0);
\coordinate (v4) at (3.0,0.0);
\coordinate (v5) at (0.0,2.0);
\coordinate (v6) at (1.0,2.0);
\coordinate (v7) at (2.0,2.0);
%\coordinate[label=right:{$v^-$}] (v8) at (1.0,-1.0);
\coordinate (v8) at (3.0,2.0);
\coordinate (v9) at (4.0,0.0);
\coordinate (v10) at (4.0,2.0);
\draw[solid,thin,color=black,-] (v1) -- (v2);
\draw[solid,thin,color=black,-] (v2) -- (v3);
\draw[solid,thin,color=black,-] (v3) -- (v4);
\draw[solid,thin,color=black,-] (v5) -- (v6);
\draw[solid,thin,color=black,-] (v6) -- (v7);
\draw[solid,thin,color=black,-] (v7) -- (v8);
\draw[solid,thin,color=black,-] (v4) -- (v9);
\draw[solid,thin,color=black,-] (v5) -- (v10);

\draw (v1) -- (v2) node[midway, below] {$\ell$};
\draw (v2) -- (v3) node[midway, below] {$t$};
\draw (v3) -- (v4) node[midway, below] {$\ell$};
\draw (v4) -- (v9) node[midway, below] {$t$};
\draw (v5) -- (v6) node[midway, above] {$\ell$};
\draw (v6) -- (v7) node[midway, above] {$t$};
\draw (v7) -- (v8) node[midway, above] {$\ell$};
\draw (v8) -- (v10) node[midway, above] {$t$};

\shade[ball color=black] (v1) circle (.05);
\shade[ball color=black] (v2) circle (.05);
\shade[ball color=black] (v3) circle (.05);
\shade[ball color=black] (v4) circle (.05);
\shade[ball color=black] (v5) circle (.05);
\shade[ball color=black] (v6) circle (.05);
\shade[ball color=black] (v7) circle (.05);
\shade[ball color=black] (v8) circle (.05);
\shade[ball color=black] (v9) circle (.05);
\shade[ball color=black] (v10) circle (.05);

\end{tikzpicture}
\begin{tikzpicture}[shift ={(0.0,0.0)},scale = 1]
\tikzstyle{every node}=[font=\tiny]

\coordinate  (v1) at (0.0,0.0);
\coordinate (v2) at (1.0,0.0);
\coordinate (v3) at (2.0,0.0);
\coordinate (v4) at (3.0,0.0);
\coordinate (v5) at (0.0,2.0);
\coordinate (v6) at (1.0,2.0);
\coordinate (v7) at (2.0,2.0);
%\coordinate[label=right:{$v^-$}] (v8) at (1.0,-1.0);
\coordinate (v8) at (3.0,2.0);
\coordinate (v9) at (4.0,0.0);
\coordinate (v10) at (4.0,2.0);
\draw[solid,thin,color=black,-] (v1) -- (v2);
\draw[solid,thin,color=black,-] (v2) -- (v3);
\draw[solid,thin,color=black,-] (v3) -- (v4);
\draw[solid,thin,color=black,-] (v5) -- (v6);
\draw[solid,thin,color=black,-] (v6) -- (v7);
\draw[solid,thin,color=black,-] (v7) -- (v8);
\draw[solid,thin,color=black,-] (v1) -- (v5);
\draw[solid,thin,color=black,-] (v4) -- (v8);

\draw (v1) -- (v2) node[midway, below] {$\ell$};
\draw (v2) -- (v3) node[midway, below] {$t$};
\draw (v3) -- (v4) node[midway, below] {$\ell$};
\draw (v5) -- (v6) node[midway, above] {$\ell$};
\draw (v6) -- (v7) node[midway, above] {$t$};
\draw (v7) -- (v8) node[midway, above] {$\ell$};
\draw (v1) -- (v5) node[midway, left] {$\mathcal{L}$};
\draw (v4) -- (v8) node[midway, left] {$\mathcal{L}$};

\shade[ball color=black] (v1) circle (.05);
\shade[ball color=black] (v2) circle (.05);
\shade[ball color=black] (v3) circle (.05);
\shade[ball color=black] (v4) circle (.05);
\shade[ball color=black] (v5) circle (.05);
\shade[ball color=black] (v6) circle (.05);
\shade[ball color=black] (v7) circle (.05);
\shade[ball color=black] (v8) circle (.05);

\end{tikzpicture}
\begin{tikzpicture}[shift ={(0.0,0.0)},scale = 1]
\tikzstyle{every node}=[font=\tiny]

\coordinate  (v1) at (0.0,0.0);
\coordinate (v2) at (1.0,0.0);
\coordinate (v3) at (2.0,0.0);
\coordinate (v4) at (3.0,0.0);
\coordinate (v5) at (0.0,2.0);
\coordinate (v6) at (1.0,2.0);
\coordinate (v7) at (2.0,2.0);
%\coordinate[label=right:{$v^-$}] (v8) at (1.0,-1.0);
\coordinate (v8) at (3.0,2.0);
\coordinate (v9) at (4.0,0.0);
\coordinate (v10) at (4.0,2.0);
\draw[solid,thin,color=black,-] (v1) -- (v2);
\draw[solid,thin,color=black,-] (v2) -- (v3);
\draw[solid,thin,color=black,-] (v3) -- (v4);
\draw[solid,thin,color=black,-] (v5) -- (v6);
\draw[solid,thin,color=black,-] (v6) -- (v7);
\draw[solid,thin,color=black,-] (v7) -- (v8);
\draw[solid,thin,color=black,-] (v1) -- (v5);
\draw[solid,thin,color=black,-] (v4) -- (v8);

\draw (v1) -- (v2) node[midway, below] {$t$};
\draw (v2) -- (v3) node[midway, below] {$\ell$};
\draw (v3) -- (v4) node[midway, below] {$t$};
\draw (v5) -- (v6) node[midway, above] {$t$};
\draw (v6) -- (v7) node[midway, above] {$\ell$};
\draw (v7) -- (v8) node[midway, above] {$t$};
\draw (v1) -- (v5) node[midway, left] {$\mathcal{L}$};
\draw (v4) -- (v8) node[midway, left] {$\mathcal{L}$};

\shade[ball color=black] (v1) circle (.05);
\shade[ball color=black] (v2) circle (.05);
\shade[ball color=black] (v3) circle (.05);
\shade[ball color=black] (v4) circle (.05);
\shade[ball color=black] (v5) circle (.05);
\shade[ball color=black] (v6) circle (.05);
\shade[ball color=black] (v7) circle (.05);
\shade[ball color=black] (v8) circle (.05);

\end{tikzpicture}
\caption{Labeled Hamiltonian cycles, respectively of type 1, 2 and 3, in the Moebius ladder}
\end{figure}
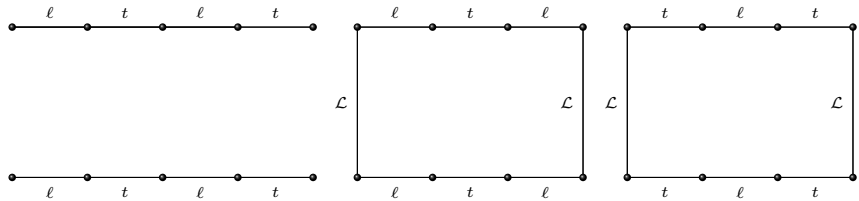

\begin{remark} (a) There exist \emph{hyperbolic} CAT(0) simplicial 2-complexes of order 2 without Hamiltonian surfaces. This is the case, for instance, of any simply connected simplicial complex with links isomorphic to  Coxeter's famous `my graph' \cite{Cox}, which is a vertex transitive  cubic non Hamiltonian graph of girth 7.   In fact, this appears to be the only known example of such a graph. (The graph is shown in Fig.\ 4.6 of \cite{God}.) It is in fact a 3-arc regular graph on 28 vertices having an automorphism group of order $336=2^4\cdot3\cdot7$ isomorphic to $\mathrm{PGL}_2(7)$. It would be interesting to study these hyperbolic complexes.

(b) It would also be interseting to study the ``locally minimally'' Hamiltonian complexes of order 2, by which we mean that every link is isomorphic to a Hamiltonian graph that contains the minimal number of, i.e., precisely three, distinct Hamiltonian cycles.
\end{remark}

We now proceed to the second step in the proof of Theorem \ref{T - main theorem}, which is ``semi-local'' in nature.

\section{Semi-local analysis}

Suppose $\Sigma$ is a Hamiltonian surface in $X$.

\begin{lemma}
 Every link of $\Sigma$ contains precisely two rungs.
 \end{lemma}

\begin{proof}
Let $x\in X$ be a vertex. If the link of $\Sigma$ at $x$ contains no rung, i.e., it is of type 1, then $\Sigma$ is locally isometric to a disk alternating in its center four triangles and four lozenges.  Here $x$ is adjacent to small lozenge angles the surface $\Sigma$. We let $e$ be an edge  not containing $x$ in any of the four triangles at $x$. Since the surface $\Sigma$ contains a lozenge incident to $e$, one vertex $y$ of the edge $e$ incident to the large angle in this lozenge. We have shown that the link of $\Sigma$ at $y$  two adjacent rungs, together an horizontal triangle edge between them (i.e., a subword of the form $\mathcal{L}-t-\mathcal{L}$ in the labeled ladder $L$).  There is no such Hamiltonian cycle in $L$ by Lemma \ref{L- Hamiltonian cycles}.
\end{proof} 

In other words, links of type 1 do not occur in $\Sigma$. In addition, a similar argument shows that links of type 2 also do not occur, and we therefore obtain the following result.

\begin{lemma}
 Every link of $\Sigma$ is of type 3.
 \end{lemma}

\begin{proof} We repeat the proof for the sake of completeness. 
Let $x\in X$ be a vertex. Suppose the link of $\Sigma$ at $x$ is of type 2, and let $D_x$ be the disk in $\Sigma$ corresponding to union of all closed faces (namely, two triangles and four lozenges) attached to $x$.  We let $e$ be an edge  not containing $x$ in one of the two triangles at $x$. Since the surface $\Sigma$ contains a lozenge incident to $e$, one vertex $y$ of the edge $e$ incident to the large angle in this lozenge. We have shown that the link of $\Sigma$ at $y$  two adjacent rungs, together an horizontal triangle edge between them (i.e., a subword of the form $\mathcal{L}-t-\mathcal{L}$ in the labeled ladder $L$).  There is no such Hamiltonian cycle in $L$ by Lemma \ref{L- Hamiltonian cycles}.
\end{proof}

We shall deduce two important preliminary results from this lemma:

\begin{theorem}
At most two of the five Hamiltonian cycles in $L$ are associated with links in a Hamiltonian surface of $X$. 
\end{theorem}

\begin{proof}
The intersection $\Sigma\cap L_x$ of a Hamiltonian surface with $L_x$ is a labeled cycle of type 3, and there are only two links of type 3 in the Moebius ladder $L$. 
\end{proof}

The next result is semi-local; it shows Hamiltonian surfaces must visit all triangles, in addition visiting all vertices and all edges.

\begin{theorem}
Every Hamiltonian surface in $X$ contains every triangle of $X$. 
\end{theorem}

\begin{proof}
Suppose $\Sigma$ is a Hamiltonian surface in $X$. For every vertex $x\in \Sigma$, the link $L_x$ is of type 3, and therefore contains the four triangles adjacent to $x$. Since $\Sigma$ visits every vertex, it contains every triangle in $X$.
\end{proof}

It follows by this result that a Hamiltonian surface must contain a checker set of lozenges  in any embedded $\diamond$-plane of $X$. Here we follow the terminology of \cite{autf2puzzles}; by $\diamond$-plane, we mean a checker type flat plane made exclusively of lozenges. We are now in a position to establish existence and uniqueness of the Hamiltonian surface.

\section{The  Hamiltonian surface}

In this section we prove the following global result. It shows that the Hamiltonian surfaces in $X$ are isolated. In fact, we shall prove more precisely that every lozenge is in a unique such surface. 

Once this fact is established, our main theorem (that $\Aut(F_2)$ is uniquely Hamiltonian) will follow by studying of the automorphism group of the Hamiltonian surface found in the present section. This is done in \ts \ref{S - proof of main t}.

\begin{theorem}\label{T - unique surface lozenge}
Every lozenge in $X$ is contained in a unique Hamiltonian surface. 
\end{theorem}

We begin with  a useful lemma.

\begin{lemma}\label{L - shuriken}
Let $t$ be a triangle and $\Sigma$ be a Hamiltonian surface in $X$. The union of $t$ and all lozenges in $\Sigma$ adjacent to $t$ forms a ``shuriken'', by which we mean the following configuration of a triangle and lozenges:

\begin{figure}[H]

\begin{tikzpicture}[line join=bevel,z=-5.5]
\tikzstyle{every node}=[font=\tiny]

\coordinate (t_0_0) at (0,0);
\coordinate (t_0_1) at (0.5,0.866);
\coordinate (t_0_2) at (1,1.732);
\coordinate (t_0_3) at (1.5,2.598);
\coordinate (t_0_4) at (2,3.464);
\coordinate (t_0_5) at (2.5,4.33);
\coordinate (t_1_0) at (1,0);
\coordinate (t_1_1) at (1.5,0.866);
\coordinate (t_1_2) at (2,1.732);
\coordinate (t_1_3) at (2.5,2.598);
\coordinate (t_1_4) at (3,3.464);
\coordinate (t_1_5) at (3.5,4.33);
\coordinate (t_2_0) at (2,0);
\coordinate (t_2_1) at (2.5,0.866);
\coordinate (t_2_2) at (3,1.732);
\coordinate (t_2_3) at (3.5,2.598);
\coordinate (t_2_4) at (4,3.464);
\coordinate (t_2_5) at (4.5,4.33);
\coordinate (t_3_0) at (3,0);
\coordinate (t_3_1) at (3.5,0.866);
\coordinate (t_3_2) at (4,1.732);
\coordinate (t_3_3) at (4.5,2.598);
\coordinate (t_3_4) at (5,3.464);
\coordinate (t_3_5) at (5.5,4.33);
\coordinate (t_4_0) at (4,0);
\coordinate (t_4_1) at (4.5,0.866);
\coordinate (t_4_2) at (5,1.732);
\coordinate (t_4_3) at (5.5,2.598);
\coordinate (t_4_4) at (6,3.464);
\coordinate (t_4_5) at (6.5,4.33);
\coordinate (t_5_0) at (5,0);
\coordinate (t_5_1) at (5.5,0.866);
\coordinate (t_5_2) at (6,1.732);
\coordinate (t_5_3) at (6.5,2.598);
\coordinate (t_5_4) at (7,3.464);
\coordinate (t_5_5) at (7.5,4.33);
\coordinate (t_5_5) at (7.5,4.33);

\draw[solid,line width=0.1mm,color=black,-] (t_1_1) -- (t_2_1) -- (t_1_2) -- cycle;
\fill [fill=black, fill opacity =0.05] (t_1_1) -- (t_2_1) -- (t_1_2) -- cycle;

\draw[solid,line width=0.1mm,color=black,-] (t_2_1) -- (t_2_0) -- (t_1_0)-- (t_1_1);
\draw[solid,line width=0.1mm,color=black,-] (t_1_1) -- (t_0_2) -- (t_0_3)-- (t_1_2);
\draw[solid,line width=0.1mm,color=black,-] (t_1_2) -- (t_2_2) -- (t_3_1)-- (t_2_1);

\end{tikzpicture}
\caption{}
\end{figure}

\noindent Furthermore, if this configuration holds locally at two vertices of the triangle, it must hold at the third vertex.
\end{lemma}

\begin{proof}
It follows by Lemma that if $\ell-t$ is a subword in the labelled link of a vertex in $\Sigma$, then $\ell-t$ extends uniquely into $\ell-t-\mathcal{L}$. This observation, applied at every vertex of the given triangle, proves the lemma. The last assertion is clear.
\end{proof}

\begin{proof}[Proof of Theorem \ref{T - unique surface lozenge}]
Lemma \ref{L - shuriken} readily implies that at most two leaves can exist around a triangle for a Hamiltonian surface in $X$. Therefore, existence and uniqueness amount to the claim that these leaves `integrate' into a pair of global Hamiltonian surface. 

In order to do so, we view $X$ as a simplicial complex (as described originally in \cite{Brady}), in which every lozenge of $X$ is viewed a union of two triangles sharing a short diagonal; thus, there are two types of edges in $X$ viewed as a simplicial complex: the edges in the boundary of a triangle (or a lozenge) which are of order 3, and the short diagonals in lozenges, which are edges of order 2.  We fix a base vertex $*\in X$, and let $B_n$ (resp.\ $S_n$) denote the simplicial ball (resp., sphere) of radius $n\geq 1$ around $*$. 

Fix a surface $\Sigma_1$ in $B_1$ corresponding to one of the two possible choices for the intersection of a Hamiltonian surface of $X$ with $B_1$. Here we may assume that $B_1$, and the surface $\Sigma_1$, contain the lozenge initially given. We shall show that $\Sigma_1$ extends uniquely into a Hamiltonian surface of $X$.

  Suppose by induction that there exists a unique Hamiltonian surface $\Sigma_n$ in $B_n$ which contains $\Sigma_1$. This is true of $n=1$, so suppose this holds for a general $n\geq 1$ and let us establish this property at $n+1$. 

If $x$ is a vertex in $S_n$, then the intersection of $B_n$ with the link of $x$ in $X$ is a simplicial tree $T_x$ of diameter at most 3. In this tree, the surface $\Sigma_n$ determines a non-backtracking path $p_x$ between two of its boundary vertices. Since this path corresponds to a Hamiltonian surface in $B_n$, it is locally structured by the Shuriken lemma; an edge with a label $t$ must be adjacent to edges having both labels $\ell$ and $\mathcal{L}$. Thus, such path is included in a path of the form $\ell-t-\mathcal{L}-t$. 

Now, there exists a unique embedding of $T_x$ in $L_x$ taking the path $p_x$ to a subpath of a Hamiltonian cycle of type 3 in $L_x$. In other words, for every vertex $x$ in $S_n$, there exists a unique extension of $\Sigma_n$ into a local surface $\Sigma_x\subset B_1(x)$ in the ball of radius 1 in $X$ around $x$. We have to prove that the diverse extensions $\Sigma_x$, where $x\in S_n$ is a vertex, are compatible.  In order to do so, it suffices to prove that any two adjacent extensions $\Sigma_x$ and $\Sigma_y$, coincide on the intersection $B_1(x)\cap B_1(y)$. 

Suppose $[x,y]$ is an edge in $S_n$. We shall show that $\Sigma_x$ and $\Sigma_y$ coincide on the open book on $[x,y]$. There are two cases to  consider. If $[x,y]$ is a diagonal of a lozenge, it is obvious that the two extensions are compatible, since the intersection $B_1(x)\cap B_1(y)$ coincide with  the corresponding lozenge, and its belonging to $\Sigma_{n+1}$ is determined by whether it was in $\Sigma_n$ or not, within the $n$-ball $B_n$. Otherwise, $[x,y]$ is an edge of order 3, which we assume from now on.

Suppose first that the face containing $[x,y]$ in $B_n$ is a lozenge. Then the triangle in $S_{n+1}$ containing $[x,y]$ belongs to both $\Sigma_x$ and $\Sigma_y$, while the lozenge in $S_{n+1}$ containing $[x,y]$ belongs to neither. This proves $\Sigma_x$ and $\Sigma_y$ coincide on the intersection  $B_1(x)\cap B_1(y)$.

Suppose next that the face containing $[x,y]$ in $B_n$ is a triangle $t$. Then the unique lozenge in $S_{n+1}$ which is adjacent to $[x,y]$ and belongs to $\Sigma_x$ is determined by Lemma \ref{L - shuriken} (last assertion). The same holds in the case $\Sigma_y$. In both cases, the chosen lozenge is the unique lozenge which complete the shuriken partially defined by the lozenges in $\Sigma_n$ which are adjacent to $t$ in $B_n$. In particular, the choice of the lozenge which belongs to $\Sigma_x$ and $\Sigma_y$ is entirely determined by the two lozenges in $\Sigma_n$ which are already attached to $t$; thus,  $\Sigma_x$ and $\Sigma_y$ coincide on on the intersection  $B_1(x)\cap B_1(y)$ in this case too.

This concludes the proof  that there exists a unique extension $\Sigma_{n+1}$ of $\Sigma_n$ into a Hamiltonian surface in $B_{n+1}$. By induction, this  implies that $\Sigma_1$ admits a  unique extension into a Hamiltonian surface $\Sigma:=\varinjlim \Sigma_n$ of the space $X$.
\end{proof}

\begin{proof}[Proof of Theorem \ref{T - main theorem}] The two Hamiltonian surfaces in the Brady complex $X$ are the surfaces attached by Theorem \ref{T - unique surface lozenge} to any pair of adjacent lozenges in $X$.
\end{proof}

\section{An explicit Hamiltonian quotient}

 Because it is canonical in $X$ (up to isomorphism), canonical compact Hamiltonian quotients of $\Sigma$ exist in every compact quotient of $X$. In this section, we describe a compact quotient $S$ of $\Sigma$ having three vertices. A construction analogous to that presented in \cite{torus}, starting from this compact hyperbolic quotient rather than a flat metric 2-torus, provides a compact space having universal cover the Brady complex, where the `knights' of \cite{torus} are now of length 4 rather than 3 (corresponding to lozenges), and the pinching step is not required.  
 
 The surface $S$ can be described in a simple way using four ``charts'' as follows:

\begin{figure}[H]

\begin{tikzpicture}[line join=bevel,z=-5.5]
\tikzstyle{every node}=[font=\tiny]

\coordinate (t_0_0) at (0,0);
\coordinate (t_0_1) at (0.5,0.866);
\coordinate (t_0_2) at (1,1.732);
\coordinate (t_0_3) at (1.5,2.598);
\coordinate (t_0_4) at (2,3.464);
\coordinate (t_0_5) at (2.5,4.33);
\coordinate (t_1_0) at (1,0);
\coordinate (t_1_1) at (1.5,0.866);
\coordinate (t_1_2) at (2,1.732);
\coordinate (t_1_3) at (2.5,2.598);
\coordinate (t_1_4) at (3,3.464);
\coordinate (t_1_5) at (3.5,4.33);
\coordinate (t_2_0) at (2,0);
\coordinate (t_2_1) at (2.5,0.866);
\coordinate (t_2_2) at (3,1.732);
\coordinate (t_2_3) at (3.5,2.598);
\coordinate (t_2_4) at (4,3.464);
\coordinate (t_2_5) at (4.5,4.33);
\coordinate (t_3_0) at (3,0);
\coordinate (t_3_1) at (3.5,0.866);
\coordinate (t_3_2) at (4,1.732);
\coordinate (t_3_3) at (4.5,2.598);
\coordinate (t_3_4) at (5,3.464);
\coordinate (t_3_5) at (5.5,4.33);
\coordinate (t_4_0) at (4,0);
\coordinate (t_4_1) at (4.5,0.866);
\coordinate (t_4_2) at (5,1.732);
\coordinate (t_4_3) at (5.5,2.598);
\coordinate (t_4_4) at (6,3.464);
\coordinate (t_4_5) at (6.5,4.33);
\coordinate (t_5_0) at (5,0);
\coordinate (t_5_1) at (5.5,0.866);
\coordinate (t_5_2) at (6,1.732);
\coordinate (t_5_3) at (6.5,2.598);
\coordinate (t_5_4) at (7,3.464);
\coordinate (t_5_5) at (7.5,4.33);
\coordinate (t_5_5) at (7.5,4.33);

%\draw[solid,line width=0.1mm,color=black,-] (t_1_1) -- (t_2_1) -- (t_1_2) -- cycle;
\fill [fill=black, fill opacity =0.05] (t_1_1) -- (t_2_1) -- (t_1_2) -- cycle;

%bottom
\draw[solid,line width=0.1mm,color=black,-latex] (t_2_1) to node[right] {$x_c$} (t_2_0);
\draw[solid,line width=0.1mm,color=black,latex-]  (t_2_0) to node[below] {$x_b$} (t_1_0);
\draw[solid,line width=0.1mm,color=black,-latex] (t_1_0) to node[left] {$x_d$} (t_1_1);
\draw[solid,line width=0.1mm,color=black,latex-] (t_1_1) to node[below] {$x_a$} (t_2_1);

%left
\draw[solid,line width=0.1mm,color=black,-latex] (t_1_1) to node[left] {$y_b$} (t_0_2);
\draw[solid,line width=0.1mm,color=black,latex-]  (t_0_2) to node[left] {$y_c$} (t_0_3);
\draw[solid,line width=0.1mm,color=black,-latex] (t_0_3) to node[right] {$y_d$} (t_1_2);
\draw[solid,line width=0.1mm,color=black,latex-] (t_1_2) to node[left] {$y_a$} (t_1_1);

%right
\draw[solid,line width=0.1mm,color=black,-latex] (t_1_2) to node[above] {$z_c$} (t_2_2);
\draw[solid,line width=0.1mm,color=black,latex-]  (t_2_2) to node[right] {$z_b$} (t_3_1);
\draw[solid,line width=0.1mm,color=black,-latex] (t_3_1) to node[below] {$z_d$} (t_2_1);
\draw[solid,line width=0.1mm,color=black,latex-] (t_2_1) to node[right] {$z_a$} (t_1_2);
\end{tikzpicture}
\begin{tikzpicture}[line join=bevel,z=-5.5]
\tikzstyle{every node}=[font=\tiny]

\coordinate (t_0_0) at (0,0);
\coordinate (t_0_1) at (0.5,0.866);
\coordinate (t_0_2) at (1,1.732);
\coordinate (t_0_3) at (1.5,2.598);
\coordinate (t_0_4) at (2,3.464);
\coordinate (t_0_5) at (2.5,4.33);
\coordinate (t_1_0) at (1,0);
\coordinate (t_1_1) at (1.5,0.866);
\coordinate (t_1_2) at (2,1.732);
\coordinate (t_1_3) at (2.5,2.598);
\coordinate (t_1_4) at (3,3.464);
\coordinate (t_1_5) at (3.5,4.33);
\coordinate (t_2_0) at (2,0);
\coordinate (t_2_1) at (2.5,0.866);
\coordinate (t_2_2) at (3,1.732);
\coordinate (t_2_3) at (3.5,2.598);
\coordinate (t_2_4) at (4,3.464);
\coordinate (t_2_5) at (4.5,4.33);
\coordinate (t_3_0) at (3,0);
\coordinate (t_3_1) at (3.5,0.866);
\coordinate (t_3_2) at (4,1.732);
\coordinate (t_3_3) at (4.5,2.598);
\coordinate (t_3_4) at (5,3.464);
\coordinate (t_3_5) at (5.5,4.33);
\coordinate (t_4_0) at (4,0);
\coordinate (t_4_1) at (4.5,0.866);
\coordinate (t_4_2) at (5,1.732);
\coordinate (t_4_3) at (5.5,2.598);
\coordinate (t_4_4) at (6,3.464);
\coordinate (t_4_5) at (6.5,4.33);
\coordinate (t_5_0) at (5,0);
\coordinate (t_5_1) at (5.5,0.866);
\coordinate (t_5_2) at (6,1.732);
\coordinate (t_5_3) at (6.5,2.598);
\coordinate (t_5_4) at (7,3.464);
\coordinate (t_5_5) at (7.5,4.33);
\coordinate (t_5_5) at (7.5,4.33);

%\draw[solid,line width=0.1mm,color=black,-] (t_1_1) -- (t_2_1) -- (t_1_2) -- cycle;
\fill [fill=black, fill opacity =0.05] (t_1_1) -- (t_2_1) -- (t_1_2) -- cycle;

%bottom
\draw[solid,line width=0.1mm,color=black,-latex] (t_2_1) to node[right] {$x_d$} (t_2_0);
\draw[solid,line width=0.1mm,color=black,latex-]  (t_2_0) to node[below] {$x_a$} (t_1_0);
\draw[solid,line width=0.1mm,color=black,-latex] (t_1_0) to node[left] {$x_c$} (t_1_1);
\draw[solid,line width=0.1mm,color=black,latex-] (t_1_1) to node[below] {$x_b$} (t_2_1);

%left
\draw[solid,line width=0.1mm,color=black,-latex] (t_1_1) to node[left] {$y_a$} (t_0_2);
\draw[solid,line width=0.1mm,color=black,latex-]  (t_0_2) to node[left] {$y_d$} (t_0_3);
\draw[solid,line width=0.1mm,color=black,-latex] (t_0_3) to node[right] {$y_c$} (t_1_2);
\draw[solid,line width=0.1mm,color=black,latex-] (t_1_2) to node[left] {$y_b$} (t_1_1);

%right
\draw[solid,line width=0.1mm,color=black,-latex] (t_1_2) to node[above] {$z_d$} (t_2_2);
\draw[solid,line width=0.1mm,color=black,latex-]  (t_2_2) to node[right] {$z_a$} (t_3_1);
\draw[solid,line width=0.1mm,color=black,-latex] (t_3_1) to node[below] {$z_c$} (t_2_1);
\draw[solid,line width=0.1mm,color=black,latex-] (t_2_1) to node[right] {$z_b$} (t_1_2);
\end{tikzpicture}

\begin{tikzpicture}[line join=bevel,z=-5.5]
\tikzstyle{every node}=[font=\tiny]

\coordinate (t_0_0) at (0,0);
\coordinate (t_0_1) at (0.5,0.866);
\coordinate (t_0_2) at (1,1.732);
\coordinate (t_0_3) at (1.5,2.598);
\coordinate (t_0_4) at (2,3.464);
\coordinate (t_0_5) at (2.5,4.33);
\coordinate (t_1_0) at (1,0);
\coordinate (t_1_1) at (1.5,0.866);
\coordinate (t_1_2) at (2,1.732);
\coordinate (t_1_3) at (2.5,2.598);
\coordinate (t_1_4) at (3,3.464);
\coordinate (t_1_5) at (3.5,4.33);
\coordinate (t_2_0) at (2,0);
\coordinate (t_2_1) at (2.5,0.866);
\coordinate (t_2_2) at (3,1.732);
\coordinate (t_2_3) at (3.5,2.598);
\coordinate (t_2_4) at (4,3.464);
\coordinate (t_2_5) at (4.5,4.33);
\coordinate (t_3_0) at (3,0);
\coordinate (t_3_1) at (3.5,0.866);
\coordinate (t_3_2) at (4,1.732);
\coordinate (t_3_3) at (4.5,2.598);
\coordinate (t_3_4) at (5,3.464);
\coordinate (t_3_5) at (5.5,4.33);
\coordinate (t_4_0) at (4,0);
\coordinate (t_4_1) at (4.5,0.866);
\coordinate (t_4_2) at (5,1.732);
\coordinate (t_4_3) at (5.5,2.598);
\coordinate (t_4_4) at (6,3.464);
\coordinate (t_4_5) at (6.5,4.33);
\coordinate (t_5_0) at (5,0);
\coordinate (t_5_1) at (5.5,0.866);
\coordinate (t_5_2) at (6,1.732);
\coordinate (t_5_3) at (6.5,2.598);
\coordinate (t_5_4) at (7,3.464);
\coordinate (t_5_5) at (7.5,4.33);
\coordinate (t_5_5) at (7.5,4.33);

%\draw[solid,line width=0.1mm,color=black,-] (t_1_1) -- (t_2_1) -- (t_1_2) -- cycle;
\fill [fill=black, fill opacity =0.05] (t_1_1) -- (t_2_1) -- (t_1_2) -- cycle;

%bottom
\draw[solid,line width=0.1mm,color=black,-latex] (t_2_1) to node[right] {$x_a$} (t_2_0);
\draw[solid,line width=0.1mm,color=black,latex-]  (t_2_0) to node[below] {$x_d$} (t_1_0);
\draw[solid,line width=0.1mm,color=black,-latex] (t_1_0) to node[left] {$x_b$} (t_1_1);
\draw[solid,line width=0.1mm,color=black,latex-] (t_1_1) to node[below] {$x_c$} (t_2_1);

%left
\draw[solid,line width=0.1mm,color=black,-latex] (t_1_1) to node[left] {$y_d$} (t_0_2);
\draw[solid,line width=0.1mm,color=black,latex-]  (t_0_2) to node[left] {$y_a$} (t_0_3);
\draw[solid,line width=0.1mm,color=black,-latex] (t_0_3) to node[right] {$y_b$} (t_1_2);
\draw[solid,line width=0.1mm,color=black,latex-] (t_1_2) to node[left] {$y_c$} (t_1_1);

%right
\draw[solid,line width=0.1mm,color=black,-latex] (t_1_2) to node[above] {$z_a$} (t_2_2);
\draw[solid,line width=0.1mm,color=black,latex-]  (t_2_2) to node[right] {$z_d$} (t_3_1);
\draw[solid,line width=0.1mm,color=black,-latex] (t_3_1) to node[below] {$z_b$} (t_2_1);
\draw[solid,line width=0.1mm,color=black,latex-] (t_2_1) to node[right] {$z_c$} (t_1_2);
\end{tikzpicture}
\begin{tikzpicture}[line join=bevel,z=-5.5]
\tikzstyle{every node}=[font=\tiny]

\coordinate (t_0_0) at (0,0);
\coordinate (t_0_1) at (0.5,0.866);
\coordinate (t_0_2) at (1,1.732);
\coordinate (t_0_3) at (1.5,2.598);
\coordinate (t_0_4) at (2,3.464);
\coordinate (t_0_5) at (2.5,4.33);
\coordinate (t_1_0) at (1,0);
\coordinate (t_1_1) at (1.5,0.866);
\coordinate (t_1_2) at (2,1.732);
\coordinate (t_1_3) at (2.5,2.598);
\coordinate (t_1_4) at (3,3.464);
\coordinate (t_1_5) at (3.5,4.33);
\coordinate (t_2_0) at (2,0);
\coordinate (t_2_1) at (2.5,0.866);
\coordinate (t_2_2) at (3,1.732);
\coordinate (t_2_3) at (3.5,2.598);
\coordinate (t_2_4) at (4,3.464);
\coordinate (t_2_5) at (4.5,4.33);
\coordinate (t_3_0) at (3,0);
\coordinate (t_3_1) at (3.5,0.866);
\coordinate (t_3_2) at (4,1.732);
\coordinate (t_3_3) at (4.5,2.598);
\coordinate (t_3_4) at (5,3.464);
\coordinate (t_3_5) at (5.5,4.33);
\coordinate (t_4_0) at (4,0);
\coordinate (t_4_1) at (4.5,0.866);
\coordinate (t_4_2) at (5,1.732);
\coordinate (t_4_3) at (5.5,2.598);
\coordinate (t_4_4) at (6,3.464);
\coordinate (t_4_5) at (6.5,4.33);
\coordinate (t_5_0) at (5,0);
\coordinate (t_5_1) at (5.5,0.866);
\coordinate (t_5_2) at (6,1.732);
\coordinate (t_5_3) at (6.5,2.598);
\coordinate (t_5_4) at (7,3.464);
\coordinate (t_5_5) at (7.5,4.33);
\coordinate (t_5_5) at (7.5,4.33);

%\draw[solid,line width=0.1mm,color=black,-] (t_1_1) -- (t_2_1) -- (t_1_2) -- cycle;
\fill [fill=black, fill opacity =0.05] (t_1_1) -- (t_2_1) -- (t_1_2) -- cycle;

%bottom
\draw[solid,line width=0.1mm,color=black,-latex] (t_2_1) to node[right] {$x_c$} (t_2_0);
\draw[solid,line width=0.1mm,color=black,latex-]  (t_2_0) to node[below] {$x_c$} (t_1_0);
\draw[solid,line width=0.1mm,color=black,-latex] (t_1_0) to node[left] {$x_a$} (t_1_1);
\draw[solid,line width=0.1mm,color=black,latex-] (t_1_1) to node[below] {$x_d$} (t_2_1);

%left
\draw[solid,line width=0.1mm,color=black,-latex] (t_1_1) to node[left] {$y_c$} (t_0_2);
\draw[solid,line width=0.1mm,color=black,latex-]  (t_0_2) to node[left] {$y_b$} (t_0_3);
\draw[solid,line width=0.1mm,color=black,-latex] (t_0_3) to node[right] {$y_a$} (t_1_2);
\draw[solid,line width=0.1mm,color=black,latex-] (t_1_2) to node[left] {$y_d$} (t_1_1);

%right
\draw[solid,line width=0.1mm,color=black,-latex] (t_1_2) to node[above] {$z_b$} (t_2_2);
\draw[solid,line width=0.1mm,color=black,latex-]  (t_2_2) to node[right] {$z_c$} (t_3_1);
\draw[solid,line width=0.1mm,color=black,-latex] (t_3_1) to node[below] {$z_a$} (t_2_1);
\draw[solid,line width=0.1mm,color=black,latex-] (t_2_1) to node[right] {$z_d$} (t_1_2);
\end{tikzpicture}
\caption{The  first compact hyperbolic Hamiltonian surface $S$}
\end{figure}
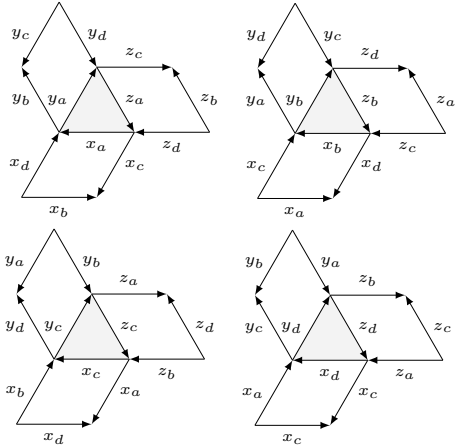

The charts indicate the local structure described in Lemma \ref{L - shuriken}. The four triangles are distinct. The surface contains three lozenges, denoted $x$, $y$ and $z$, which are adjacent to each triangle. The charts describe the local gluing pattern at every vertex of every triangle (all charts are used at each vertex). It is straightforward to verify that the links are circles of length $10\pi/3$, and that the surface is compact hyperbolic of genus two.

We can repeat the  procedure in \cite{torus} (and reviewed in the introduction of the present paper). This is technically simpler in the present paper, since the pinching step is now irrelevant, and it suffices to add only three lozenges to the surface $S$ in order to obtain a compact space $V$ with universal cover the Brady complex (however, the connexion to surgery is lost, and one cannot recover, for instance, a result such as Theorem 1.3 in \cite{torus}).

The three new lozenges are straightforward to compute. They are given by:

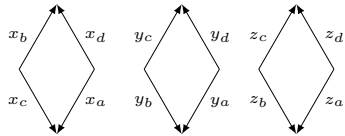
\begin{figure}[H]

\begin{tikzpicture}[line join=bevel,z=-5.5]
\tikzstyle{every node}=[font=\tiny]

\coordinate (t_0_0) at (0,0);
\coordinate (t_0_1) at (0.5,0.866);
\coordinate (t_0_2) at (1,1.732);
\coordinate (t_0_3) at (1.5,2.598);
\coordinate (t_0_4) at (2,3.464);
\coordinate (t_0_5) at (2.5,4.33);
\coordinate (t_1_0) at (1,0);
\coordinate (t_1_1) at (1.5,0.866);
\coordinate (t_1_2) at (2,1.732);
\coordinate (t_1_3) at (2.5,2.598);
\coordinate (t_1_4) at (3,3.464);
\coordinate (t_1_5) at (3.5,4.33);
\coordinate (t_2_0) at (2,0);
\coordinate (t_2_1) at (2.5,0.866);
\coordinate (t_2_2) at (3,1.732);
\coordinate (t_2_3) at (3.5,2.598);
\coordinate (t_2_4) at (4,3.464);
\coordinate (t_2_5) at (4.5,4.33);
\coordinate (t_3_0) at (3,0);
\coordinate (t_3_1) at (3.5,0.866);
\coordinate (t_3_2) at (4,1.732);
\coordinate (t_3_3) at (4.5,2.598);
\coordinate (t_3_4) at (5,3.464);
\coordinate (t_3_5) at (5.5,4.33);
\coordinate (t_4_0) at (4,0);
\coordinate (t_4_1) at (4.5,0.866);
\coordinate (t_4_2) at (5,1.732);
\coordinate (t_4_3) at (5.5,2.598);
\coordinate (t_4_4) at (6,3.464);
\coordinate (t_4_5) at (6.5,4.33);
\coordinate (t_5_0) at (5,0);
\coordinate (t_5_1) at (5.5,0.866);
\coordinate (t_5_2) at (6,1.732);
\coordinate (t_5_3) at (6.5,2.598);
\coordinate (t_5_4) at (7,3.464);
\coordinate (t_5_5) at (7.5,4.33);
\coordinate (t_5_5) at (7.5,4.33);

%left
\draw[solid,line width=0.1mm,color=black,latex-] (t_1_1) to node[left] {$x_c$} (t_0_2);
\draw[solid,line width=0.1mm,color=black,-latex]  (t_0_2) to node[left] {$x_b$} (t_0_3);
\draw[solid,line width=0.1mm,color=black,latex-] (t_0_3) to node[right] {$x_d$} (t_1_2);
\draw[solid,line width=0.1mm,color=black,-latex] (t_1_2) to node[right] {$x_a$} (t_1_1);
\end{tikzpicture}
\begin{tikzpicture}[line join=bevel,z=-5.5]
\tikzstyle{every node}=[font=\tiny]

\coordinate (t_0_0) at (0,0);
\coordinate (t_0_1) at (0.5,0.866);
\coordinate (t_0_2) at (1,1.732);
\coordinate (t_0_3) at (1.5,2.598);
\coordinate (t_0_4) at (2,3.464);
\coordinate (t_0_5) at (2.5,4.33);
\coordinate (t_1_0) at (1,0);
\coordinate (t_1_1) at (1.5,0.866);
\coordinate (t_1_2) at (2,1.732);
\coordinate (t_1_3) at (2.5,2.598);
\coordinate (t_1_4) at (3,3.464);
\coordinate (t_1_5) at (3.5,4.33);
\coordinate (t_2_0) at (2,0);
\coordinate (t_2_1) at (2.5,0.866);
\coordinate (t_2_2) at (3,1.732);
\coordinate (t_2_3) at (3.5,2.598);
\coordinate (t_2_4) at (4,3.464);
\coordinate (t_2_5) at (4.5,4.33);
\coordinate (t_3_0) at (3,0);
\coordinate (t_3_1) at (3.5,0.866);
\coordinate (t_3_2) at (4,1.732);
\coordinate (t_3_3) at (4.5,2.598);
\coordinate (t_3_4) at (5,3.464);
\coordinate (t_3_5) at (5.5,4.33);
\coordinate (t_4_0) at (4,0);
\coordinate (t_4_1) at (4.5,0.866);
\coordinate (t_4_2) at (5,1.732);
\coordinate (t_4_3) at (5.5,2.598);
\coordinate (t_4_4) at (6,3.464);
\coordinate (t_4_5) at (6.5,4.33);
\coordinate (t_5_0) at (5,0);
\coordinate (t_5_1) at (5.5,0.866);
\coordinate (t_5_2) at (6,1.732);
\coordinate (t_5_3) at (6.5,2.598);
\coordinate (t_5_4) at (7,3.464);
\coordinate (t_5_5) at (7.5,4.33);
\coordinate (t_5_5) at (7.5,4.33);

%left
\draw[solid,line width=0.1mm,color=black,latex-] (t_1_1) to node[left] {$y_b$} (t_0_2);
\draw[solid,line width=0.1mm,color=black,-latex]  (t_0_2) to node[left] {$y_c$} (t_0_3);
\draw[solid,line width=0.1mm,color=black,latex-] (t_0_3) to node[right] {$y_d$} (t_1_2);
\draw[solid,line width=0.1mm,color=black,-latex] (t_1_2) to node[right] {$y_a$} (t_1_1);
\end{tikzpicture}\begin{tikzpicture}[line join=bevel,z=-5.5]
\tikzstyle{every node}=[font=\tiny]

\coordinate (t_0_0) at (0,0);
\coordinate (t_0_1) at (0.5,0.866);
\coordinate (t_0_2) at (1,1.732);
\coordinate (t_0_3) at (1.5,2.598);
\coordinate (t_0_4) at (2,3.464);
\coordinate (t_0_5) at (2.5,4.33);
\coordinate (t_1_0) at (1,0);
\coordinate (t_1_1) at (1.5,0.866);
\coordinate (t_1_2) at (2,1.732);
\coordinate (t_1_3) at (2.5,2.598);
\coordinate (t_1_4) at (3,3.464);
\coordinate (t_1_5) at (3.5,4.33);
\coordinate (t_2_0) at (2,0);
\coordinate (t_2_1) at (2.5,0.866);
\coordinate (t_2_2) at (3,1.732);
\coordinate (t_2_3) at (3.5,2.598);
\coordinate (t_2_4) at (4,3.464);
\coordinate (t_2_5) at (4.5,4.33);
\coordinate (t_3_0) at (3,0);
\coordinate (t_3_1) at (3.5,0.866);
\coordinate (t_3_2) at (4,1.732);
\coordinate (t_3_3) at (4.5,2.598);
\coordinate (t_3_4) at (5,3.464);
\coordinate (t_3_5) at (5.5,4.33);
\coordinate (t_4_0) at (4,0);
\coordinate (t_4_1) at (4.5,0.866);
\coordinate (t_4_2) at (5,1.732);
\coordinate (t_4_3) at (5.5,2.598);
\coordinate (t_4_4) at (6,3.464);
\coordinate (t_4_5) at (6.5,4.33);
\coordinate (t_5_0) at (5,0);
\coordinate (t_5_1) at (5.5,0.866);
\coordinate (t_5_2) at (6,1.732);
\coordinate (t_5_3) at (6.5,2.598);
\coordinate (t_5_4) at (7,3.464);
\coordinate (t_5_5) at (7.5,4.33);
\coordinate (t_5_5) at (7.5,4.33);

%left
\draw[solid,line width=0.1mm,color=black,latex-] (t_1_1) to node[left] {$z_b$} (t_0_2);
\draw[solid,line width=0.1mm,color=black,-latex]  (t_0_2) to node[left] {$z_c$} (t_0_3);
\draw[solid,line width=0.1mm,color=black,latex-] (t_0_3) to node[right] {$z_d$} (t_1_2);
\draw[solid,line width=0.1mm,color=black,-latex] (t_1_2) to node[right] {$z_a$} (t_1_1);
\end{tikzpicture}
\caption{The three filling lozenges $x'$, $y'$, and $z'$}\label{Fig - filling S}
\end{figure}

Observe that each new lozenge contains two of the three vertices of the surface $S$; for each vertex, this new set of lozenge adds two edges with label $\mathcal{L}$, and two edges with label $\ell$, in order to form the Moebius ladder $L$. Thus, we obtain, as announced:

\begin{theorem}
The universal cover $\widetilde V$ is the Brady complex $X$. 
\end{theorem}

\begin{proof}
This follows directly by \cite[Theorem 4.3]{torus}.
\end{proof}

There is a sort of duality (``arrow revering'') that leads to symmetry and balance between the surface $S$ and its sibling $S'$, in relation with the ambiant complex $V$. The second Hamiltonian surface $S'$ consists of the same set of four triangles together with the lozenges shown in Fig.\ \ref{Fig - filling S}. The intersection $S\cap S'$ in $V$ is the set of triangles, as is the case in the universal cover $X$. We prove in the next section that $S$ and $S'$ are isomorphic through an automorphism of $V$.

\begin{remark}
Although not required to prove the theorems stated in the introduction, it is interesting in view of \cite{torus} to compute in the space $V$ what the quotients of the $\diamond$-flats are in $X$. In fact, $V$ can be viewed as a space built from these quotients, using four triangles, and three flat surfaces, each of which is a union of two lozenges. The computation of these flat surfaces is straightforward. Namely, it is obvious that the union of $x$ and $x'$ is a flat 2-torus, as is the union of $y$ and $y'$. The union of $z$ and $z'$, on the other hand, is  a Klein bottle and thus is non-orientable. This gives an example of a (non-orientable) torsion-free subgroup $\pi_1(V)$ of finite index in $\Aut(F_2)$ constructed from $t=2$ tori and $b=1$ Klein Bottle (compare \cite[Remark 5.3]{torus}).
\end{remark}

Having explicit genus 2 quotients of our Hamitonian surfaces, we can now establish the last result of our paper. 

\section{Proof of Theorem \ref{T - isomorphic}}\label{S - proof of main t}

We shall exhibit automorphisms of $V$ using the description found in the previous \ts; it results that the two Hamiltonian surfaces are equivariantly isomorphic. 

We know that, in addition to the four triangles,  $V$ contains six lozenges as follows:

\begin{figure}[H]

\begin{tikzpicture}[line join=bevel,z=-5.5]
\tikzstyle{every node}=[font=\tiny]

\coordinate (t_0_0) at (0,0);
\coordinate (t_0_1) at (0.5,0.866);
\coordinate (t_0_2) at (1,1.732);
\coordinate (t_0_3) at (1.5,2.598);
\coordinate (t_0_4) at (2,3.464);
\coordinate (t_0_5) at (2.5,4.33);
\coordinate (t_1_0) at (1,0);
\coordinate (t_1_1) at (1.5,0.866);
\coordinate (t_1_2) at (2,1.732);
\coordinate (t_1_3) at (2.5,2.598);
\coordinate (t_1_4) at (3,3.464);
\coordinate (t_1_5) at (3.5,4.33);
\coordinate (t_2_0) at (2,0);
\coordinate (t_2_1) at (2.5,0.866);
\coordinate (t_2_2) at (3,1.732);
\coordinate (t_2_3) at (3.5,2.598);
\coordinate (t_2_4) at (4,3.464);
\coordinate (t_2_5) at (4.5,4.33);
\coordinate (t_3_0) at (3,0);
\coordinate (t_3_1) at (3.5,0.866);
\coordinate (t_3_2) at (4,1.732);
\coordinate (t_3_3) at (4.5,2.598);
\coordinate (t_3_4) at (5,3.464);
\coordinate (t_3_5) at (5.5,4.33);
\coordinate (t_4_0) at (4,0);
\coordinate (t_4_1) at (4.5,0.866);
\coordinate (t_4_2) at (5,1.732);
\coordinate (t_4_3) at (5.5,2.598);
\coordinate (t_4_4) at (6,3.464);
\coordinate (t_4_5) at (6.5,4.33);
\coordinate (t_5_0) at (5,0);
\coordinate (t_5_1) at (5.5,0.866);
\coordinate (t_5_2) at (6,1.732);
\coordinate (t_5_3) at (6.5,2.598);
\coordinate (t_5_4) at (7,3.464);
\coordinate (t_5_5) at (7.5,4.33);
\coordinate (t_5_5) at (7.5,4.33);

%left
\draw[solid,line width=0.1mm,color=black,-latex] (t_1_1) to node[left] {$x_c$} (t_0_2);
\draw[solid,line width=0.1mm,color=black,latex-]  (t_0_2) to node[left] {$x_b$} (t_0_3);
\draw[solid,line width=0.1mm,color=black,-latex] (t_0_3) to node[right] {$x_d$} (t_1_2);
\draw[solid,line width=0.1mm,color=black,latex-] (t_1_2) to node[right] {$x_a$} (t_1_1);
\end{tikzpicture}
\begin{tikzpicture}[line join=bevel,z=-5.5]
\tikzstyle{every node}=[font=\tiny]

\coordinate (t_0_0) at (0,0);
\coordinate (t_0_1) at (0.5,0.866);
\coordinate (t_0_2) at (1,1.732);
\coordinate (t_0_3) at (1.5,2.598);
\coordinate (t_0_4) at (2,3.464);
\coordinate (t_0_5) at (2.5,4.33);
\coordinate (t_1_0) at (1,0);
\coordinate (t_1_1) at (1.5,0.866);
\coordinate (t_1_2) at (2,1.732);
\coordinate (t_1_3) at (2.5,2.598);
\coordinate (t_1_4) at (3,3.464);
\coordinate (t_1_5) at (3.5,4.33);
\coordinate (t_2_0) at (2,0);
\coordinate (t_2_1) at (2.5,0.866);
\coordinate (t_2_2) at (3,1.732);
\coordinate (t_2_3) at (3.5,2.598);
\coordinate (t_2_4) at (4,3.464);
\coordinate (t_2_5) at (4.5,4.33);
\coordinate (t_3_0) at (3,0);
\coordinate (t_3_1) at (3.5,0.866);
\coordinate (t_3_2) at (4,1.732);
\coordinate (t_3_3) at (4.5,2.598);
\coordinate (t_3_4) at (5,3.464);
\coordinate (t_3_5) at (5.5,4.33);
\coordinate (t_4_0) at (4,0);
\coordinate (t_4_1) at (4.5,0.866);
\coordinate (t_4_2) at (5,1.732);
\coordinate (t_4_3) at (5.5,2.598);
\coordinate (t_4_4) at (6,3.464);
\coordinate (t_4_5) at (6.5,4.33);
\coordinate (t_5_0) at (5,0);
\coordinate (t_5_1) at (5.5,0.866);
\coordinate (t_5_2) at (6,1.732);
\coordinate (t_5_3) at (6.5,2.598);
\coordinate (t_5_4) at (7,3.464);
\coordinate (t_5_5) at (7.5,4.33);
\coordinate (t_5_5) at (7.5,4.33);

%left
\draw[solid,line width=0.1mm,color=black,-latex] (t_1_1) to node[left] {$y_b$} (t_0_2);
\draw[solid,line width=0.1mm,color=black,latex-]  (t_0_2) to node[left] {$y_c$} (t_0_3);
\draw[solid,line width=0.1mm,color=black,-latex] (t_0_3) to node[right] {$y_d$} (t_1_2);
\draw[solid,line width=0.1mm,color=black,latex-] (t_1_2) to node[right] {$y_a$} (t_1_1);
\end{tikzpicture}\begin{tikzpicture}[line join=bevel,z=-5.5]
\tikzstyle{every node}=[font=\tiny]

\coordinate (t_0_0) at (0,0);
\coordinate (t_0_1) at (0.5,0.866);
\coordinate (t_0_2) at (1,1.732);
\coordinate (t_0_3) at (1.5,2.598);
\coordinate (t_0_4) at (2,3.464);
\coordinate (t_0_5) at (2.5,4.33);
\coordinate (t_1_0) at (1,0);
\coordinate (t_1_1) at (1.5,0.866);
\coordinate (t_1_2) at (2,1.732);
\coordinate (t_1_3) at (2.5,2.598);
\coordinate (t_1_4) at (3,3.464);
\coordinate (t_1_5) at (3.5,4.33);
\coordinate (t_2_0) at (2,0);
\coordinate (t_2_1) at (2.5,0.866);
\coordinate (t_2_2) at (3,1.732);
\coordinate (t_2_3) at (3.5,2.598);
\coordinate (t_2_4) at (4,3.464);
\coordinate (t_2_5) at (4.5,4.33);
\coordinate (t_3_0) at (3,0);
\coordinate (t_3_1) at (3.5,0.866);
\coordinate (t_3_2) at (4,1.732);
\coordinate (t_3_3) at (4.5,2.598);
\coordinate (t_3_4) at (5,3.464);
\coordinate (t_3_5) at (5.5,4.33);
\coordinate (t_4_0) at (4,0);
\coordinate (t_4_1) at (4.5,0.866);
\coordinate (t_4_2) at (5,1.732);
\coordinate (t_4_3) at (5.5,2.598);
\coordinate (t_4_4) at (6,3.464);
\coordinate (t_4_5) at (6.5,4.33);
\coordinate (t_5_0) at (5,0);
\coordinate (t_5_1) at (5.5,0.866);
\coordinate (t_5_2) at (6,1.732);
\coordinate (t_5_3) at (6.5,2.598);
\coordinate (t_5_4) at (7,3.464);
\coordinate (t_5_5) at (7.5,4.33);
\coordinate (t_5_5) at (7.5,4.33);

%left
\draw[solid,line width=0.1mm,color=black,-latex] (t_1_1) to node[left] {$z_b$} (t_0_2);
\draw[solid,line width=0.1mm,color=black,latex-]  (t_0_2) to node[left] {$z_c$} (t_0_3);
\draw[solid,line width=0.1mm,color=black,-latex] (t_0_3) to node[right] {$z_a$} (t_1_2);
\draw[solid,line width=0.1mm,color=black,latex-] (t_1_2) to node[right] {$z_d$} (t_1_1);
\end{tikzpicture}

\begin{tikzpicture}[line join=bevel,z=-5.5]
\tikzstyle{every node}=[font=\tiny]

\coordinate (t_0_0) at (0,0);
\coordinate (t_0_1) at (0.5,0.866);
\coordinate (t_0_2) at (1,1.732);
\coordinate (t_0_3) at (1.5,2.598);
\coordinate (t_0_4) at (2,3.464);
\coordinate (t_0_5) at (2.5,4.33);
\coordinate (t_1_0) at (1,0);
\coordinate (t_1_1) at (1.5,0.866);
\coordinate (t_1_2) at (2,1.732);
\coordinate (t_1_3) at (2.5,2.598);
\coordinate (t_1_4) at (3,3.464);
\coordinate (t_1_5) at (3.5,4.33);
\coordinate (t_2_0) at (2,0);
\coordinate (t_2_1) at (2.5,0.866);
\coordinate (t_2_2) at (3,1.732);
\coordinate (t_2_3) at (3.5,2.598);
\coordinate (t_2_4) at (4,3.464);
\coordinate (t_2_5) at (4.5,4.33);
\coordinate (t_3_0) at (3,0);
\coordinate (t_3_1) at (3.5,0.866);
\coordinate (t_3_2) at (4,1.732);
\coordinate (t_3_3) at (4.5,2.598);
\coordinate (t_3_4) at (5,3.464);
\coordinate (t_3_5) at (5.5,4.33);
\coordinate (t_4_0) at (4,0);
\coordinate (t_4_1) at (4.5,0.866);
\coordinate (t_4_2) at (5,1.732);
\coordinate (t_4_3) at (5.5,2.598);
\coordinate (t_4_4) at (6,3.464);
\coordinate (t_4_5) at (6.5,4.33);
\coordinate (t_5_0) at (5,0);
\coordinate (t_5_1) at (5.5,0.866);
\coordinate (t_5_2) at (6,1.732);
\coordinate (t_5_3) at (6.5,2.598);
\coordinate (t_5_4) at (7,3.464);
\coordinate (t_5_5) at (7.5,4.33);
\coordinate (t_5_5) at (7.5,4.33);

%left
\draw[solid,line width=0.1mm,color=black,latex-] (t_1_1) to node[left] {$x_c$} (t_0_2);
\draw[solid,line width=0.1mm,color=black,-latex]  (t_0_2) to node[left] {$x_b$} (t_0_3);
\draw[solid,line width=0.1mm,color=black,latex-] (t_0_3) to node[right] {$x_d$} (t_1_2);
\draw[solid,line width=0.1mm,color=black,-latex] (t_1_2) to node[right] {$x_a$} (t_1_1);
\end{tikzpicture}
\begin{tikzpicture}[line join=bevel,z=-5.5]
\tikzstyle{every node}=[font=\tiny]

\coordinate (t_0_0) at (0,0);
\coordinate (t_0_1) at (0.5,0.866);
\coordinate (t_0_2) at (1,1.732);
\coordinate (t_0_3) at (1.5,2.598);
\coordinate (t_0_4) at (2,3.464);
\coordinate (t_0_5) at (2.5,4.33);
\coordinate (t_1_0) at (1,0);
\coordinate (t_1_1) at (1.5,0.866);
\coordinate (t_1_2) at (2,1.732);
\coordinate (t_1_3) at (2.5,2.598);
\coordinate (t_1_4) at (3,3.464);
\coordinate (t_1_5) at (3.5,4.33);
\coordinate (t_2_0) at (2,0);
\coordinate (t_2_1) at (2.5,0.866);
\coordinate (t_2_2) at (3,1.732);
\coordinate (t_2_3) at (3.5,2.598);
\coordinate (t_2_4) at (4,3.464);
\coordinate (t_2_5) at (4.5,4.33);
\coordinate (t_3_0) at (3,0);
\coordinate (t_3_1) at (3.5,0.866);
\coordinate (t_3_2) at (4,1.732);
\coordinate (t_3_3) at (4.5,2.598);
\coordinate (t_3_4) at (5,3.464);
\coordinate (t_3_5) at (5.5,4.33);
\coordinate (t_4_0) at (4,0);
\coordinate (t_4_1) at (4.5,0.866);
\coordinate (t_4_2) at (5,1.732);
\coordinate (t_4_3) at (5.5,2.598);
\coordinate (t_4_4) at (6,3.464);
\coordinate (t_4_5) at (6.5,4.33);
\coordinate (t_5_0) at (5,0);
\coordinate (t_5_1) at (5.5,0.866);
\coordinate (t_5_2) at (6,1.732);
\coordinate (t_5_3) at (6.5,2.598);
\coordinate (t_5_4) at (7,3.464);
\coordinate (t_5_5) at (7.5,4.33);
\coordinate (t_5_5) at (7.5,4.33);

%left
\draw[solid,line width=0.1mm,color=black,latex-] (t_1_1) to node[left] {$y_b$} (t_0_2);
\draw[solid,line width=0.1mm,color=black,-latex]  (t_0_2) to node[left] {$y_c$} (t_0_3);
\draw[solid,line width=0.1mm,color=black,latex-] (t_0_3) to node[right] {$y_d$} (t_1_2);
\draw[solid,line width=0.1mm,color=black,-latex] (t_1_2) to node[right] {$y_a$} (t_1_1);
\end{tikzpicture}\begin{tikzpicture}[line join=bevel,z=-5.5]
\tikzstyle{every node}=[font=\tiny]

\coordinate (t_0_0) at (0,0);
\coordinate (t_0_1) at (0.5,0.866);
\coordinate (t_0_2) at (1,1.732);
\coordinate (t_0_3) at (1.5,2.598);
\coordinate (t_0_4) at (2,3.464);
\coordinate (t_0_5) at (2.5,4.33);
\coordinate (t_1_0) at (1,0);
\coordinate (t_1_1) at (1.5,0.866);
\coordinate (t_1_2) at (2,1.732);
\coordinate (t_1_3) at (2.5,2.598);
\coordinate (t_1_4) at (3,3.464);
\coordinate (t_1_5) at (3.5,4.33);
\coordinate (t_2_0) at (2,0);
\coordinate (t_2_1) at (2.5,0.866);
\coordinate (t_2_2) at (3,1.732);
\coordinate (t_2_3) at (3.5,2.598);
\coordinate (t_2_4) at (4,3.464);
\coordinate (t_2_5) at (4.5,4.33);
\coordinate (t_3_0) at (3,0);
\coordinate (t_3_1) at (3.5,0.866);
\coordinate (t_3_2) at (4,1.732);
\coordinate (t_3_3) at (4.5,2.598);
\coordinate (t_3_4) at (5,3.464);
\coordinate (t_3_5) at (5.5,4.33);
\coordinate (t_4_0) at (4,0);
\coordinate (t_4_1) at (4.5,0.866);
\coordinate (t_4_2) at (5,1.732);
\coordinate (t_4_3) at (5.5,2.598);
\coordinate (t_4_4) at (6,3.464);
\coordinate (t_4_5) at (6.5,4.33);
\coordinate (t_5_0) at (5,0);
\coordinate (t_5_1) at (5.5,0.866);
\coordinate (t_5_2) at (6,1.732);
\coordinate (t_5_3) at (6.5,2.598);
\coordinate (t_5_4) at (7,3.464);
\coordinate (t_5_5) at (7.5,4.33);
\coordinate (t_5_5) at (7.5,4.33);

%left
\draw[solid,line width=0.1mm,color=black,latex-] (t_1_1) to node[left] {$z_b$} (t_0_2);
\draw[solid,line width=0.1mm,color=black,-latex]  (t_0_2) to node[left] {$z_c$} (t_0_3);
\draw[solid,line width=0.1mm,color=black,latex-] (t_0_3) to node[right] {$z_d$} (t_1_2);
\draw[solid,line width=0.1mm,color=black,-latex] (t_1_2) to node[right] {$z_a$} (t_1_1);
\end{tikzpicture}
\caption{The six lozenges $x$, $y$, $z$ and (below them) $x'$, $y'$, and $z'$}\label{Fig - filling S}
\end{figure}
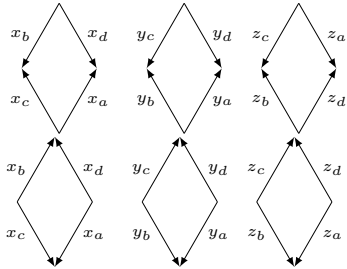

An automorphism permutes this six elements set and the triangles. We may define several.

A first apparent map, although not one suitable for our purpose of proving that $\Sigma$ and $\Sigma'$ are isomorphic, is:

\begin{align*}
x_a&\lra x_d\\
x_b&\lra x_c\\
y_a&\lra y_d\\
y_b&\lra x_c\\
z_a&\lra z_d\\
z_b&\lra z_c\\
\end{align*}

We call this map $\theta_1$. It is obvious to verify that this is an automorphism of $V$ which takes triangles $a$ and $b$ to $d$ and $c$ respectively. It flips every lozenge along the long diagonal.

A second map $\theta_2$ acts as follows:

\begin{align*}
x_a&\lra y_a^{-1}\\
x_c&\lra y_b^{-1}\\
x_b&\lra y_c^{-1}\\
x_d&\lra y_d^{-1}\\
z_a&\lra z_a^{-1}\\
z_b&\lra z_c^{-1}\\
z_c&\lra z_b^{-1}\\
z_d&\lra z_d^{-1}\\
\end{align*}

Thus,  it takes $x$ to $y'$ and $y$ to $x'$, and permutes $z$ and $z'$. This maps acts on the triangles $a$ and $b$, and permutes the triangles $c$ and $d$. It is also straightforward to verify that it induces an automorphism of $V$.

\begin{lemma}
There exits an involutive automorphism of $V$ taking $S$ to $S'$.
\end{lemma}

\begin{proof}
$\theta_2$ is such an automorphism.
\end{proof}

We may now prove our second main result. 

\begin{proof}[Proof of Theorem \ref{T - isomorphic}]
This follows directly from the previous lemma lifting by lifting $\theta_2$ to the universal cover. This map $\tilde \theta_2$ permutes the two Hamiltonian surfaces in $X$ and  $\tilde \theta_2\circ s=(\theta_2)_*(s)\circ \tilde \theta_2$ where $s\in \pi_1(B)$. The group $\pi_1(B)$ is a finite index subgroup of $\Aut(F_2)$ viewed as a acting on $X$ isometrically.
\end{proof}

We note that $\theta_1$ and $\theta_2$ commute. In fact, the following holds:

\begin{theorem}
$\Aut(B)=\ZI/2\ZI^3$.
\end{theorem}

\begin{proof}
A map $\theta\in \Aut(B)$ induces an automorphism of the Klein bottle $K:=\ip{z,z'}$. Furthermore, after composing with $\theta_2$, we may assume that $\theta$ stabilizes $z$. Consider the following map which we call $\theta_3$:

\begin{align*}
x_a&\lra x_b\\
x_c&\lra x_d\\
y_a&\lra y_b\\
y_c&\lra y_d\\
z_a&\lra z_b\\
z_c&\lra z_d\\
\end{align*}

This map is an automorphism of $V$ which rotates every lozenge. After composing with $\theta_1$ and $\theta_2$, we may assume that $\theta$ acts trivially on $z$ and therefore is trivial. Thus, $\Aut(B)$ is a group isomorphic to $\ZI/2\ZI^3$ generated by three involutive automorphisms $\theta_1$, $\theta_2$ and $\theta_3$.
\end{proof}


\begin{thebibliography}{00}



\bibitem{torus}

Barr\'e, S.; Pichot, M., Surgery on $\Aut(F_2)$, Bulletin of the Belgian
Mathematical Society - Simon Stevin Vol. 30, Issue 1, pgs 31-50, 2023.


\bibitem{autf2puzzles} Barré S.,  Pichot M., $\Aut(F_2)$ puzzles. Geometriae Dedicata, 199(1), 225-246, 2019.

\bibitem{Brady}
Brady T., Automatic structures on $\Aut(F_2)$, Archiv der Math., 63(2)/97–102, 1994. 


\bibitem{cp}
    Crisp J., Paoluzzi, L.  On the classification of CAT(0) structures for the 4-string braid   group. Mich. Math. J., 53(1)/133–163, 2005. 

\bibitem{BH}
Bridson, M., Haefliger, A. Metric spaces of non-positive curvature, Springer Berlin, Heidelberg (1999)


\bibitem{Cox}
Coxeter, H.S.M., My graph, Proc. London Math. Soc. 46 (1983) 117-136.



\bibitem{God}
Godsil, C., and Royle, G., Algebraic Graph Theory, Graduate Texts in Mathematics 207 (Springer, 2001).


\bibitem{Tu} 
Tutte, W.T., On Hamiltonian circuits, J. London Math. Soc., 21 (1946), 98--101.

\end{thebibliography}
\end{document}